\documentclass{amsart}
\usepackage{amsfonts}

\newtheorem{theorem}{Theorem}[section]
\newtheorem{corollary}[theorem]{Corollary}

\newtheorem{proposition}[theorem]{Proposition}
\newtheorem{remark}[theorem]{Remark}

\numberwithin{theorem}{section}
\input{tcilatex}

\begin{document}
\title[Operator factorization of range space relations ]{Operator
factorization of range space relations }
\author{Maria Joi\c{t}a }
\address{Department of Mathematics\\
Faculty of Applied Sciences, University Politehnica of Bucharest, 313 Spl.
Independentei, 060042, Bucharest, Romania}
\email{mjoita@fmi.unibuc.ro and maria.joita@mathem.pub.ro}
\author{Tania Lumini\c{t}a Costache }
\address{Department of Mathematics\\
Faculty of Applied Sciences, University Politehnica of Bucharest, 313 Spl.
Independentei, 060042, Bucharest, Romania}
\email{lumycos1@yahoo.com}
\thanks{This paper is in final form and no version of it will be submitted
for publication elsewhere.}

\begin{abstract}
Given two range space relations $A$ and $B$ in Hilbert spaces, we
characterize the existence of a range space operator $T$ such that $A=BT,$
respectively $A=TB.$
\end{abstract}

\maketitle

\section{Introduction and Preliminaries}

Arens \cite{A} introduced the notion of linear relation in linear spaces to
extend the theory of single-valued operator to the multi-valued case.

In what follows we remind some notions and results from \cite{A, Sa, FW, SS}
that we need throughout this paper.

Let $\mathcal{X}$ and $\mathcal{Y}$ be two linear spaces. A subspace $%
A\subseteq $ $\mathcal{X}$ $\times $ $\mathcal{Y}$ \ of $\mathcal{X}$ $%
\times $ $\mathcal{Y}$ is called a linear relation in $\mathcal{X}$ $\times $
$\mathcal{Y}$. The notations dom$(A)$ and ran$(A)$ denote the domain of $A$,
respectively the range of $A$, which are linear subspaces of $\mathcal{X}$,
respectively of $\mathcal{Y}$ defined by 
\begin{equation*}
\text{dom}\left( A\right) =\{x\in \mathcal{X};\left( \exists \right) y\in 
\mathcal{Y}\text{ such that }\left( x,y\right) \in A\},
\end{equation*}%
respectively 
\begin{equation*}
\text{ran}\left( A\right) =\{y\in \mathcal{Y};\left( \exists \right) x\in 
\mathcal{X}\text{ such that }\left( x,y\right) \in A\}.
\end{equation*}
The notations $\ker (A)$ and mul$(A)$ denote the kernel of $A$, respectively
the multi-valued part of $A$, which are linear subspaces of $\mathcal{X}$,
respectively of $\mathcal{Y}$ defined by 
\begin{equation*}
\ker \left( A\right) =\{x\in \mathcal{X};\text{ }\left( x,0\right) \in A\},
\end{equation*}%
respectively 
\begin{equation*}
\text{mul}\left( A\right) =\{y\in \mathcal{Y};\text{ }\left( 0,y\right) \in
A\}.
\end{equation*}

A linear relation $A$ is the graph of a linear operator if and only if mul$%
\left( A\right) =\{0\},$ and a linear operator will be identified with its
graph.

If $A$ $\subseteq $ $\mathcal{X}$ $\times $ $\mathcal{Y}$ \ is a linear
relation, then $A^{-1}$ $\subseteq $ $\mathcal{Y}\times \mathcal{X}$ is a
linear relation with dom$\left( A^{-1}\right) =$ran$\left( A\right) ,$ ran$%
\left( A^{-1}\right) =$dom$\left( A\right) ,\ker \left( A^{-1}\right) =$mul$%
\left( A\right) $ and mul$\left( A^{-1}\right) =\ker \left( A\right) .$

Let $\mathcal{X},\mathcal{Y}$ and $\mathcal{Z}$ be three linear spaces, $%
A\subseteq $ $\mathcal{X}$ $\times $ $\mathcal{Y}$ \ and $B\subseteq $ $%
\mathcal{Y}$ $\times $ $\mathcal{Z}$ \ two linear relations. The product $BA$
defined by 
\begin{equation*}
BA=\{(x,z);(\exists )y\in \mathcal{Y}\text{ such that }\left( x,y\right) \in
A\text{ and }\left( y,z\right) \in \mathcal{Z}\}
\end{equation*}%
is a linear relation in $\mathcal{X}$ $\times $ $\mathcal{Z}$.

Let $\mathcal{X}_{i},$ and $\mathcal{Y}_{i},$ $i=1,2$ be four linear spaces
and $A_{i}\subseteq $ $\mathcal{X}_{i}$ $\times $ $\mathcal{Y}_{i},$ $i=1,2$ 
$\ $be two linear relations. The component -wise sum $A_{1}\widehat{+}A_{2}$
defined by 
\begin{equation*}
A_{1}\widehat{+}A_{2}=\{\left( x_{1}+x_{2},y_{1}+y_{2}\right) ;\left(
x_{1},y_{1}\right) \in A_{1},\left( x_{2},y_{2}\right) \in A_{2}\}
\end{equation*}%
is a linear relation and it is a direct sum if $A_{1}\cap A_{2}=\{\left(
0,0\right) \}.$

An operator range or a range space operator in a Hilbert space $\mathcal{X}$
is a linear subspace of $\mathcal{X}$ which is the range of some bounded
linear operator on $\mathcal{X}$ \cite{FW}. So the closed subspaces of $%
\mathcal{X}$ are range space operators.

The notion of range space relation in Hilbert spaces was introduced by A.
Sandovici as a generalization of the notion of range space operator (see, 
\cite[Section 6.4]{Sa}). Let $\mathcal{X}$ and $\mathcal{Y}$ be two Hilbert
spaces and $A\subseteq $ $\mathcal{X}$ $\times $ $\mathcal{Y}$ be a linear
relation. We say that $A$ is a range space relation if there is a Hilbert
space $\mathcal{Z}$ and a bounded linear operator $F:$ $\mathcal{%
Z\rightarrow X}$ $\times $ $\mathcal{Y}$ such that ran$\left( F\right) =A$.

\begin{remark}
\label{rem}

\begin{enumerate}
\item Let $\mathcal{X}$ and $\mathcal{Y}$ be two Hilbert spaces and $%
A\subseteq $ $\mathcal{X}$ $\times $ $\mathcal{Y}$ a range space relation.
Then $A^{-1}\subseteq $ $\mathcal{Y\times X}$ is a range space relation \cite%
[6.4.1, p. 131]{Sa}.

\item Let $\mathcal{X},$ $\mathcal{Y}$ and $\mathcal{Z}$ be three separable
Hilbert spaces, $A\subseteq $ $\mathcal{X}$ $\times $ $\mathcal{Y}$ and $%
B\subseteq $ $\mathcal{Y}\times \mathcal{Z}$ two range space relations. Then 
$BA$ is a range space relation in $\mathcal{X}$ $\times $ $\mathcal{Z}$%
\textbf{\ }\cite[Proposition 6.4.8]{Sa}.

\item Let $\mathcal{X}_{i},$ and $\mathcal{Y}_{i},$ $i=1,2$ be four Hilbert
spaces and $A_{i}\subseteq $ $\mathcal{X}_{i}$ $\times $ $\mathcal{Y}_{i},$ $%
i=1,2$ $\ $be two range space relations. Then $A_{1}\widehat{+}A_{2}$ is a
range space relation \cite[p.133]{Sa}.
\end{enumerate}
\end{remark}

Given a pair $\left( A,B\right) $ of bounded linear operators on a Hilbert
space $\mathcal{X}$,\ Douglas \cite{D} characterized the existence of a
bounded linear operator $T$ such that $A=BT$ in terms of the range of $A$
and the range of $B$. More precisely, he proved that there is a bounded
linear operator \ $T$ on $\mathcal{X}$ such that $A=BT$ if and only if ran$%
(A)\subseteq $\ ran$(B)$.

Later, Popovici and Sebesty\'{e}n \cite{PS} and Sandovici and Sebesty\'{e}n 
\cite{SS} extended Douglas' result in context of linear relations. Sandovici
and Sebesty\'{e}n \cite{SS} proved some factorization problems of linear
relation by showing that if $A$\ and $B$\ are two linear relations in linear
spaces then there is a linear operator $T$\ such that $A=BT$\ if and only if
ran$(A)\subseteq $\ ran$(B)$\ and mul$(A)=$mul$(B)$\ \cite[Theorem 1]{SS}
and also that there is a linear operator $T$\textbf{\ }such that $A=TB$\ if%
\textbf{\ }and only if dom$(A)\subseteq $dom$(B)$, ker$(B)\subseteq $ker$(A)$%
\ and there is a linear subspace\textbf{\ }$\mathcal{D}\subseteq $mul$(B)$\
and a surjective linear operator\textbf{\ }$T_{\text{mul}}:\mathcal{D}%
\longrightarrow \text{mul}(A)$ \cite[Theorem 2]{SS}\textbf{.}

This paper concerns the factorization problems of range space relations.
More exactly, we show that given two range space relations $A\subseteq $ $%
\mathcal{X}$ $\times $ $\mathcal{Z}$ and $B\subseteq $ $\mathcal{Y}\times 
\mathcal{Z}$, $T=B^{-1}A$ is a range space operator solution of the equation 
$A=BX$ if and only if ran$\left( A\right) \subseteq $ran$\left( B\right) ,$
mul$\left( B\right) =$mul$\left( A\right) $ and $B^{-1}$is the graph of a
linear operator, and if $A\subseteq $ $\mathcal{X}$ $\times $ $\mathcal{Y}$
and $B\subseteq $ $\mathcal{X}\times \mathcal{Z}$, $T=AB^{-1}$ is a range
space operator solution of the equation $A=XB$ if and only if dom$\left(
A\right) \subseteq $dom $\left( B\right) $, $\ker \left( B\right) \subseteq
\ker \left( A\right) $ and $A$ is the graph of a linear operator. We also
obtain, under some hypotheses regarding $A$ and $B$, a necessary and
sufficient condition for the existence of a bounded linear operator $T$ $%
\subseteq \mathcal{X}$ $\times $ $\mathcal{Y}$ such that $A=BT$ for $%
A\subseteq $ $\mathcal{X}$ $\times $ $\mathcal{Z}$ and $B\subseteq $ $%
\mathcal{Y}\times \mathcal{Z\ }$(Theorem \ref{T1}), respectively, for the
existence of a bounded linear operator $T\subseteq \mathcal{Z}$ $\times $ $%
\mathcal{Y}$ such that $A=TB$,$\ $ for $A\subseteq $ $\mathcal{X}$ $\times $ 
$\mathcal{Y}$ and $B\subseteq $ $\mathcal{X}\times \mathcal{Z\ }$(Theorem %
\ref{T2}). The adjoint $A^{\ast }$ of linear relation $A\subseteq \mathcal{X}
$ $\times $ $\mathcal{X}$ in a Hilbert space $\mathcal{X}$ is a range space
relation \cite{A, SAN}. In Proposition \ref{P21}, we obtain a necessary and
sufficient condition for the existence of a bounded linear operator $T$ $%
\subseteq \mathcal{X}$ $\times $ $\mathcal{X}$ such that $A^{\ast }=B^{\ast
}T$ , respectively, for the existence of a bounded linear operator $%
T\subseteq \mathcal{X}$ $\times $ $\mathcal{X}$ such that $A^{\ast
}=TB^{\ast }$ in terms of $A$ and $B$.

\section{Main results}

\bigskip

Let $\mathcal{X},\mathcal{Y}$ and $\mathcal{Z}$ be three Hilbert spaces.
Given two range space relations $A\subseteq $ $\mathcal{X}$ $\times $ $%
\mathcal{Z}$ and $B\subseteq $ $\mathcal{Y}\times \mathcal{Z},$ as in the
case of linear relation \cite[Corollary 2.4]{PS}, we obtain a necessary and
sufficient condition such that $C=B^{-1}A$ is a solution of the equation $%
A=BX.$

\begin{proposition}
\label{P1} Let $\mathcal{X},$ $\mathcal{Y}$ and $\mathcal{Z}$ be three
Hilbert spaces and let $A\subseteq $ $\mathcal{X}$ $\times $ $\mathcal{Z}$
and $B\subseteq $ $\mathcal{Y}\times \mathcal{Z}$ be two range space
relations. Then $C=B^{-1}A$ is a solution in the set of all the range space
relations of the equation $A=BX$ if and only if ran$\left( A\right)
\subseteq $ran$\left( B\right) $ and mul$\left( B\right) \subseteq $mul$%
\left( A\right) .$
\end{proposition}

\begin{proof}
It follows from \cite[Corollary 2.4]{PS} and Remark \ref{rem}.
\end{proof}

Let $A\subseteq $ $\mathcal{X}$ $\times $ $\mathcal{Z}$ and $B\subseteq $ $%
\mathcal{Y}\times \mathcal{Z}$ be two linear relations. If ran$\left(
A\right) \subseteq $ran$\left( B\right) $, then $B^{-1}A$ is a linear
operator if and only if mul$\left( A\right) \subseteq $mul$\left( B\right) $
and $\ker (B)=\{0\}$ \cite[Section3, p. 44]{PS}. Using the above results and
taking into account Proposition \ref{P1} we obtain the following corollary.

\begin{corollary}
Let $\mathcal{X},$ $\mathcal{Y}$ and $\mathcal{Z}$ be three Hilbert spaces
and let $A\subseteq $ $\mathcal{X}$ $\times $ $\mathcal{Z}$ and $B\subseteq $
$\mathcal{Y}\times \mathcal{Z}$ be two range space relations. Then $%
T=B^{-1}A $ is a range space operator solution of the equation $A=BX$ if and
only if ran$\left( A\right) \subseteq $ran$\left( B\right) $, mul$\left(
B\right) =$mul$\left( A\right) $ and $\ker (B)=\{0\}.$
\end{corollary}

Let $\mathcal{X}$ and $\mathcal{Y}$ be two Hilbert spaces and $A\subseteq $ $%
\mathcal{X}$ $\times $ $\mathcal{Y}$ a range space relation. Then there is a
Hilbert space $\mathcal{Z}$ and a bounded linear operator $F:$ $\mathcal{%
Z\rightarrow X}$ $\times $ $\mathcal{Y}$ such that ran$\left( F\right) =A$.
We can suppose that $F$ is injective. Thus, $A\ $is a Hilbert space with the
inner product given by $\left\langle \left( x_{1},y_{1}\right) ,\left(
x_{2},y_{2}\right) \right\rangle _{A}=\left\langle h_{1},h_{2}\right\rangle $%
, where $\left( x_{i},y_{i}\right) =F\left( h_{i}\right) $, $i=1,2$.
Moreover, 
\begin{equation*}
\frac{1}{\sqrt{2}}\left( \left\Vert x\right\Vert +\left\Vert y\right\Vert
\right) \leq \left\Vert \left( x,y\right) \right\Vert _{A}
\end{equation*}%
for all $\left( x,y\right) \in A$. Then 
\begin{equation*}
J_{A}=\{\left( \left( x,y\right) ,x\right) ;\left( x,y\right) \in
A\}\subseteq \mathcal{X}\times \mathcal{Y\times X}
\end{equation*}%
is the graph of a linear operator with dom$\left( J_{A}\right) =A\ $and ran$%
\left( J_{A}\right) =$dom$\left( A\right) $.$\ $We remark that $\ker
(J_{A})=\{0\}\times $mul$(A)$. By \cite[Proposition 6.4.3]{Sa}, $%
J_{A}:A\rightarrow $dom$\left( A\right) ,$ $J_{A}\left( \left( x,y\right)
\right) =x$ is a surjective bounded linear operator, and if dom$\left(
A\right) $ is closed, $J_{A}$ is an open map.

\begin{remark}
Suppose that dom$(A)$ and mul$\left( A\right) $ are closed subspaces in $%
\mathcal{X}$, respectively $\mathcal{Y}$.$\ $Then, $A\subseteq $ $\mathcal{X}
$ $\times $ $\mathcal{Y}$ is closed \cite[Lemma 6.4.6]{Sa}, and 
\begin{equation*}
A=\left( \{0\}\times \text{mul}(A)\right) \oplus \left( A\cap \left( 
\mathcal{X}\times \text{mul}(A)^{\bot }\right) \right) .
\end{equation*}%
Therefore, the map $\widetilde{J_{A}}:$dom$\left( A\right) \rightarrow $ $A,$
given by $\widetilde{J_{A}}\left( x\right) =\left( x,y\right) $ with $y$ $%
\in $mul$\left( A\right) ^{\bot }$ is an injective bounded linear operator.
Moreover, $J_{A}\circ \widetilde{J_{A}}=$id$_{\text{dom}\left( A\right) }.$
\end{remark}

In the following theorem, under some hypotheses regarding $\ A$ and $B$, we
obtain a necessary and sufficient condition for the existence of a bounded
linear operator $T$ such that $A=BT.$

\begin{theorem}
\label{T1}Let $\mathcal{X},$ $\mathcal{Y}$ and $\mathcal{Z}$ be three
separable Hilbert spaces, $A\subseteq $ $\mathcal{X}$ $\times $ $\mathcal{Z}$
and $B\subseteq $ $\mathcal{Y}\times \mathcal{Z}$ two range space relations
such that dom$\left( A\right) $, ran$\left( B\right) $ and $\ker (B)$ are
closed. Then there is a bounded linear operator $T\subseteq \mathcal{X}%
\mathcal{\times }$ $\mathcal{Y}$ such that $A=BT$ if and only if ran$\left(
A\right) \subseteq $ran$\left( B\right) $ and mul$\left( B\right) =$mul$%
\left( A\right) .$
\end{theorem}

\begin{proof}
$"\Rightarrow "$ It follows from \cite[Theorem 5.3.1]{Sa}.

$"\Leftarrow "\ $By \cite[Theorem 5.3.1]{Sa}, there is a graph relation $%
T\subseteq \mathcal{X}$ $\times $ $\mathcal{Y}$ such that $A=BT$. By the
proof of \cite[Theorem 5.3.1]{Sa}, dom$\left( T\right) =$dom$\left( A\right) 
$ and ran$\left( T\right) \subseteq $ $\mathcal{Y}_{B},$ where $\mathcal{Y}%
_{B}$ is the orthogonal complement of $\ker (B)$. Moreover, $\left(
x,y_{B}\right) \in T$ if and only if there is $z\in $ran$\left( A\right) $
such that $\left( y_{B},z\right) \in B$.

Since ran$\left( B\right) $ is closed, mul$\left( B\right) $ is closed \cite[%
Corollary 6.4.5]{Sa}, and since mul$\left( B\right) $ $=$ mul$\left(
A\right) $, mul$\left( A\right) $ is closed. Let $\mathcal{Z=}$mul$\left(
A\right) \oplus \mathcal{Z}_{A}$.

Since dom$\left( A\right) $ is closed, the map $\widetilde{J_{A}}:$dom$%
\left( A\right) \rightarrow A$ defined by $\widetilde{J_{A}}\left( x\right)
=\left( x,z\right) $ with $z\in \mathcal{Z}_{A}$, is a bounded linear
operator.

Let $\left\{ x_{n}\right\} _{n}$ $\subseteq $dom$\left( A\right) $ be a
sequence convergent to $x$. Since $\widetilde{J_{A}}\ $is continuous, the
sequence $\left\{ \widetilde{J_{A}}\left( x_{n}\right) \right\} _{n}\ $%
converges to $\widetilde{J_{A}}\left( x\right) $, and if $\widetilde{J_{A}}%
\left( x_{n}\right) =\left( x_{n},z_{nA}\right) $ and $\widetilde{J_{A}}%
\left( x\right) =\left( x,z_{A}\right) ,$ then the sequence $\left\{
z_{nA}\right\} _{n}$ converges to $z_{A}.$

On the other hand, since ran$\left( B\right) $ is closed, and $\ker (B)$ is
closed, the linear map $\widetilde{J_{B^{-1}}}:$ran$\left( B\right)
\rightarrow $ $B^{-1},$ $\widetilde{J_{B^{-1}}}\left( z\right) =\left(
z,y\right) \ $with $y\in \mathcal{Y}_{B},$ is continuous. Since ran$\left(
A\right) \subseteq $ran$\left( B\right) ,$ $z_{nA},z_{A}\in $ran$\left(
B\right) $ for all $n\in \mathbb{N}$. Therefore there is a sequence $%
\{y_{n}\}_{n}\subseteq $ dom$\left( B\right) $ and $y\in $dom$\left(
B\right) $ such that $\left( y_{n},z_{nA}\right) \in B$ for all $n\in 
\mathbb{N}$ and $(y,z_{A})\in B$. If for each $n\in \mathbb{N},$ $%
y_{n}=y_{nB}+y_{n\ker (B)}$ and $y=y_{B}+y_{\ker (B)}$, then $(y_{n\ker
(B)},0)\in B$ for all $n\in \mathbb{N}$ and $\left( y_{\ker \left( B\right)
},0\right) \in B.$ Moreover, 
\begin{equation*}
\left( y_{nB},z_{nA}\right) =\left( y_{n},z_{nA}\right) -\left( y_{n\ker
\left( B\right) },0\right) \in B
\end{equation*}%
and 
\begin{equation*}
\left( y_{B},z_{A}\right) =\left( y,z_{A}\right) -\left( y_{\ker \left(
B\right) },0\right) \in B.
\end{equation*}

Then $\widetilde{J_{B^{-1}}}\left( z_{A}\right) =\left( z_{A},y_{B}\right) $
and for each $n\in \mathbb{N}$ and $\widetilde{J_{B^{-1}}}\left(
z_{nA}\right) =\left( z_{nA},y_{nB}\right) $. Since the sequence $\left\{
z_{nA}\right\} _{n}$ converges to $z_{A}\ $and $\widetilde{J_{B^{-1}}}$ is
continuous, the sequence $\left\{ \left( z_{nA},y_{nB}\right) \right\} _{n}$
converges $\ $to $\left( z_{A},y_{B}\right) $ in $B^{-1}$, and so the
sequence $\left\{ y_{nB}\right\} _{n}$ converges to $y_{B}$. 
%Therefore $\left( y_{B},z_{A}\right) \in B$.

Thus, we showed that for each $n\in \mathbb{N},$ $\left( x_{n},y_{nB}\right)
\in T,\left( x,y_{B}\right) \in T$ and $y_{B}=\lim\limits_{n}y_{nB}$.
Therefore, $T$ is a bounded linear operator from dom$(T)$ to $\mathcal{Y}$.
\end{proof}

\begin{corollary}
Let $\mathcal{X},$ $\mathcal{Y}$ and $\mathcal{Z}$ be three separable
Hilbert spaces and let $A\subseteq $ $\mathcal{X}$ $\times $ $\mathcal{Z}$
and $B\subseteq $ $\mathcal{Y}\times \mathcal{Z}$ be two range space
operators such that dom$\left( A\right) $, ran$\left( B\right) $ and $\ker
(B)$ are closed. Then there is a bounded linear operator $T\subseteq 
\mathcal{X}\mathcal{\times }$ $\mathcal{Y}$ such that $A=BT$\ if and only if
ran$\left( A\right) \subseteq $ran$\left( B\right) .$
\end{corollary}

The following result is \cite[Corollary 2.5]{PS} in the context of range
space relations.

\begin{proposition}
\label{P2} Let $\mathcal{X},$ $\mathcal{Y}$ and $\mathcal{Z}$ be three
Hilbert spaces and let $A\subseteq $ $\mathcal{X}$ $\times $ $\mathcal{Y}$
and $B\subseteq $ $\mathcal{X}\times \mathcal{Z}$ be two range space
relations. Then $C=AB^{-1}$ is a solution in the set of all the range space
relations of the equation $A=XB$ if and only if dom$\left( A\right)
\subseteq $dom$\left( B\right) $ and ker$\left( B\right) \subseteq \ker
\left( A\right) .$
\end{proposition}

\begin{proof}
It follows from \cite[Corollary 2.5]{PS} and Remark \ref{rem}.
\end{proof}

Let $A\subseteq $ $\mathcal{X}$ $\times $ $\mathcal{Y}$ and $B\subseteq $ $%
\mathcal{X}\times \mathcal{Z}$ two range linear relations, then $AB^{-1}$ is
an operator if and only if dom$\left( A\right) \subseteq $dom$\left(
B\right) ,$ $\ker \left( B\right) \cap $dom$\left( A\right) \subseteq \ker
\left( A\right) $ and mul$(A)=\{0\}$ \cite[Section 4, p.46]{PS}.

Using the above results and taking into account Proposition \ref{P2} we
obtain the following corollary.

\begin{corollary}
Let $\mathcal{X},$ $\mathcal{Y}$ and $\mathcal{Z}$ be three Hilbert spaces, $%
A\subseteq $ $\mathcal{X}$ $\times $ $\mathcal{Y}$ and $B\subseteq $ $%
\mathcal{X}\times \mathcal{Z}$ two range space relations. Then $T=AB^{-1}$
is a range space operator solution of the equation $A=XB$ if and only if dom$%
\left( A\right) \subseteq $dom$\left( B\right) $, ker$\left( B\right)
\subseteq $ker$\left( A\right) $ and mul$(A)=\{0\}.$
\end{corollary}

In the following theorem, under some hypotheses regarding $\ A$ and $B$, we
obtain a necessary and sufficient condition for the existence of a bounded
linear operator $T$ such that $A=TB.$

\begin{theorem}
\label{T2}Let $\mathcal{X},$ $\mathcal{Y}$ and $\mathcal{Z}$ be three
separable Hilbert spaces, $A\subseteq $ $\mathcal{X}$ $\times $ $\mathcal{Y}$
and $B\subseteq $ $\mathcal{X}$ $\times $ $\mathcal{Z}$ two range space
relations such that dom$\left( A\right) ,$ mul$\left( A\right) ,$ ran$\left(
B\right) $ and $\ker (B)$ are closed. Then, there is a bounded linear
operator $T\subseteq \mathcal{Z\times }$ $\mathcal{Y}$ \ such that $A=TB$ if
and only if \ \ \ \ \ 

\begin{enumerate}
\item dom$\left( A\right) \subseteq $dom$\left( B\right) ;$

\item $\ker \left( B\right) \subseteq \ker \left( A\right) ;$

\item there is a range space operator $T_{\text{mul}}\subseteq $mul$\left(
B\right) \mathcal{\times }$mul$\left( A\right) $ such that dom$\left( T_{%
\text{mul}}\right) $ is a closed subspace of mul$\left( B\right) $ and ran$%
\left( T_{\text{mul}}\right) =$mul$\left( A\right) $.
\end{enumerate}
\end{theorem}

\begin{proof}
$"\Rightarrow "$ From \cite[Theorem 5.3.2]{Sa} it follows that dom$\left(
A\right) \subseteq $dom$\left( B\right) $ and $\ \ker \left( B\right) $ $%
\subseteq \ker \left( A\right) $ and there is a surjective linear operator $%
T_{\text{mul}}:$mul$\left( B\right) \cap $dom$\left( T\right) \rightarrow $%
mul$\left( A\right) .$ Moreover, $T_{\text{mul }}=\left. T\right\vert _{%
\text{mul}\left( B\right) \cap \text{dom}\left( T\right) }$. Since ran$(B)$
is closed, mul$\left( B\right) $\ is closed \cite[Corollary 6.4.5]{Sa}, and
then dom$\left( T_{\text{mul }}\right) =$mul$\left( B\right) \cap $dom$%
\left( T\right) $ is a closed subspace of $\mathcal{Z}$.$\ $Since $T$:dom$%
\left( T\right) \rightarrow \mathcal{Y}$ is a bounded linear operator, $T_{%
\text{mul}}:$mul$\left( B\right) \cap $dom$\left( T\right) $ $\rightarrow $%
mul$\left( A\right) $ is a bounded linear operator. Therefore, $T_{\text{mul}%
}\subseteq $mul$\left( B\right) \mathcal{\times }$mul$\left( A\right) $ is a
range space operator with dom$\left( T_{\text{mul}}\right) $ a closed
subspace of mul$\left( B\right) $ and ran$\left( T_{\text{mul}}\right) $ $=$%
mul$\left( A\right) $.

$"\Leftarrow "$ Since mul$\left( A\right) $ and mul$\left( B\right) $ are
closed subspaces, $\mathcal{Y}=$mul$\left( A\right) \oplus \mathcal{Y}_{A}$
and $\mathcal{Z}=$mul$\left( B\right) \oplus \mathcal{Z}_{B},$ where $%
\mathcal{Y}_{A}=$mul$\left( A\right) ^{\perp }$ and $\mathcal{Z}_{B}=$mul$%
\left( B\right) ^{\perp }$. By the proof of \cite[Theorem 5.3.2]{Sa}, 
\begin{equation*}
T_{0}=\{\left( z_{B},y_{A}\right) ;\left( \exists \right) x\in \mathcal{X\ }%
\text{such that }\left( x,y_{A}\right) \in A,\left( x,z_{B}\right) \in
B\}\subseteq \mathcal{Z}_{B}\times \mathcal{Y}_{A}
\end{equation*}%
is the graph of a linear operator.

We remark that $\left( z_{B},y_{A}\right) \in T_{0}$ if and only if there is 
$x_{B}\in \ker \left( B\right) ^{\perp }$such that $\left(
x_{B},y_{A}\right) \in A$ and $\left( x_{B},z_{B}\right) \in B$.

Indeed, since $\ker (B)$ is closed, $\mathcal{X=}\ker \left( B\right) \oplus
\ker \left( B\right) ^{\perp }$.

Let $x=x_{0}+x_{B}$\ with $x_{0}\in \ker (B)$ and $x_{B}\in \ker \left(
B\right) ^{\perp }$ such that $\left( x,y_{A}\right) \in A$ and $\left(
x,z_{B}\right) \in B$. Since $\ker \left( B\right) \subseteq \ker \left(
A\right) $, $\left( x_{0},0\right) \in A$ and we deduce that 
\begin{equation*}
\left( x_{B},y_{A}\right) =\left( x,y_{A}\right) -\left( x_{0},0\right) \in A
\end{equation*}%
and 
\begin{equation*}
\left( x_{B},z_{B}\right) =\left( x,z_{B}\right) -\left( x_{0},0\right) \in
B.
\end{equation*}%
Conversely, if $\left( x_{B},y_{A}\right) \in A$ and $\left(
x_{B},z_{B}\right) \in B$, then, clearly, $\left( x,y_{A}\right) \in A$ and $%
\left( x,z_{B}\right) \in B$.

Clearly, dom$\left( T_{0}\right) \subseteq $ran$\left( B\right) $. Let $%
\left\{ z_{nB}\right\} _{n}$ be a sequence of elements in dom$(T_{0})$ which
converges to $z_{B}$. Since ran$\left( B\right) $ is closed, $z_{B}\in $ran$%
\left( B\right) $, and since dom$\left( B^{-1}\right) $ $=$ran$(B),$ the
linear map $\widetilde{J_{B^{-1}}}:$ran$\left( B\right) \rightarrow B^{-1}\ $%
given by $\widetilde{J_{B^{-1}}}\left( z\right) =\left( z,x\right) ,\ $with $%
x\in $mul$\left( B^{-1}\right) ^{\bot }=\ker \left( B\right) ^{\bot },$ is
continuous.

For each $n\in \mathbb{N}$, let $x_{n}\in \ker \left( B\right) ^{\bot }$
such that $\widetilde{J_{B^{-1}}}\left( z_{nB}\right) =\left(
z_{nB},x_{n}\right) $. Then the sequence $\left\{ \left(
z_{_{n}B},x_{n}\right) \right\} _{n}$ $\subseteq B^{-1}$ converges to $%
\left( z_{B},x\right) \ \ $in $B^{-1}$, and so the sequence $\left\{
x_{n}\right\} _{n}$ converges to $x$.

On the other hand, since $\left\{ z_{nB}\right\} _{n}$ is a sequence of
elements in dom$(T_{0})$, there is a sequence $\left\{ x_{nB}\right\} _{n} $
in $\ker \left( B\right) ^{\perp }$ and a sequence $\left\{ y_{nA}\right\}
_{n}\ $in $\mathcal{Y}_{A}$ such that for all $n\in \mathbb{N},$ $\left(
x_{nB},z_{nB}\right) \in B$ and $\left( x_{nB},y_{nA}\right) \in A $. From 
\begin{equation*}
\left( x_{n},z_{nB}\right) -\left( x_{nB},z_{nB}\right) =\left(
x_{n}-x_{nB},0\right) \in B\ \text{for all }n\in \mathbb{N}
\end{equation*}%
we conclude that, for each $n\in \mathbb{N},x_{n}-x_{nB}\in \ker \left(
B\right) $, but $x_{n}-x_{nB}\in \ker (B)^{\bot },\ $and so $x_{n}=x_{nB}$ .
Therefore, the sequence $\left\{ x_{nB}\right\} _{n}$ converges to $x,$ and
since dom$\left( A\right) $ is closed,$\ x\in $dom$\left( A\right) $.

Since dom$\left( A\right) \ $is closed, the linear map $\widetilde{J_{A}}:$%
dom$\left( A\right) \rightarrow A$ defined by $\widetilde{J_{A}}\left(
x\right) =\left( x,y\right) ,\ $with $y\in \mathcal{Y}_{A},$ is continuos.
For each $n\in \mathbb{N}$, let $y_{n}\in \mathcal{Y}_{A}\ $such\ that $%
\widetilde{J_{A}}\left( x_{nB}\right) =\left( x_{nB},y_{n}\right) $. Then
the sequence $\left\{ \left( x_{nB},y_{_{n}}\right) \right\} _{n}$ converges
to $\left( x_{B},y\right) $ in $A$, and so the sequence $\left\{
y_{n}\right\} _{n}\ \ $converges to $y$. From 
\begin{equation*}
\left( 0,y_{_{n}}-y_{nA}\right) =\left( x_{nB},y_{n}\right) -\left(
x_{nB},y_{nA}\right) \in A\text{\ for all }n\in \mathbb{N}
\end{equation*}%
we conclude that for each $n\in \mathbb{N}$, $y_{_{n}}-y_{nA}\in $mul$(A)$,
but $y_{_{n}}-y_{nA}\in $mul$\left( A\right) ^{\bot }$, and so $%
y_{_{n}}=y_{nA}$. Therefore, the sequence $\left\{ y_{nA}\right\} _{n}\ \ $%
converges to $y$.

Thus, we showed that there is $x\in \ker \left( B\right) ^{\bot }$ and $y\in 
\mathcal{Y}_{A}$ such that $\left( x,z_{B}\right) \in B$ and $\left(
x,y\right) \in A$. Therefore, $\left( z_{B},y\right) $ $\in T_{0}$. This
means that dom$(T_{0})$ is closed and $T_{0}$ is a bounded linear operator
from dom$(T_{0})$ to $\mathcal{Y}_{A}$. Therefore, $T_{0}\subseteq \mathcal{Z%
}_{B}\times \mathcal{Y}_{A}$ is a bounded linear operator.

Then $T=T_{0}\oplus T_{\text{mul }}\subseteq \mathcal{Z\times Y}$ is a
bounded linear operator. Moreover, $A=BT$ (see\ the proof of \cite[Theorem
5.3.2]{Sa}).
\end{proof}

\begin{corollary}
Let $\mathcal{X},$ $\mathcal{Y}$ and $\mathcal{Z}$ be three separable
Hilbert spaces and let $A\subseteq $ $\mathcal{X}$ $\times $ $\mathcal{Y}$
and $B\subseteq $ $\mathcal{X}$ $\times $ $\mathcal{Z}$ be two range space
operators such that dom$\left( A\right) ,$ ran$\left( B\right) $ and $\ker
(B)$ are closed. Then there is a bounded linear operator $T\subseteq 
\mathcal{Z}$ $\times $ $\mathcal{Y}$ such that $A=TB$ if and only if dom$%
\left( A\right) \subseteq $dom$\left( B\right) $ and $\ker \left( B\right)
\subseteq \ker \left( A\right) $.
\end{corollary}

Let $\mathcal{X}$ be a Hilbert space and let $A\subseteq $ $\mathcal{X}$ $%
\times $ $\mathcal{X}$ be a linear relation. The adjoint of $A$ is the
linear relation defined by 
\begin{equation*}
A^{\ast }=\{\left( x,y\right) \in \mathcal{X}\times \mathcal{X};\left\langle
y,u\right\rangle =\left\langle x,v\right\rangle \text{ for all }\left(
u,v\right) \in A\}.
\end{equation*}%
The subspace $A^{\ast \ }$is closed and so $A^{\ast \ }$is a range space
relation. Moreover,%
\begin{equation*}
\text{mul}\left( A^{\ast }\right) =\text{dom}\left( A\right) ^{\bot }\text{
and }\ker \left( A^{\ast }\right) =\text{ran}\left( A\right) ^{\bot }.
\end{equation*}%
Therefore, mul$\left( A^{\ast }\right) $and $\ker \left( A^{\ast }\right) $
are closed.

A linear relation $A\subseteq \mathcal{X}$ $\times $ $\mathcal{X}$ is
self-adjoint if $A=A^{\ast }$ \cite[p.29]{SAN}.

\begin{corollary}
Let $\mathcal{X}$ be a separable Hilbert space and let $A\subseteq $ $%
\mathcal{X}$ $\times $ $\mathcal{X}$ and $B\subseteq $ $\mathcal{X}$ $\times 
$ $\mathcal{X}$ be two self-adjoint linear relations such that dom$\left(
A\right) \ $and ran$\left( B\right) $ are closed. Then:

\begin{enumerate}
\item There is a bounded linear operator $T\subseteq \mathcal{X}$ $\times $ $%
\mathcal{X}$ such that $A=BT$ if and only if ran$\left( A\right) \subseteq $%
ran$\left( B\right) $ and mul$\left( B\right) =$mul$\left( A\right) .$

\item There is a bounded linear operator $T\subseteq \mathcal{X}$ $\times $ $%
\mathcal{X}$ such that $A=TB$ if and only if dom$\left( A\right) \subseteq $%
dom$\left( B\right) $, $\ker \left( B\right) \subseteq \ker \left( A\right) $
and there is a range space operator $T_{\text{mul}}\subseteq $mul$\left(
B\right) \mathcal{\times }$mul$\left( A\right) $ such that dom$\left( T_{%
\text{mul}}\right) $ is a closed subspace in mul$\left( B\right) $ and ran$%
\left( T_{\text{mul}}\right) =$mul$\left( A\right) $.
\end{enumerate}
\end{corollary}

\begin{proof}
Since $A$ and $B$ are self-adjoint linear relations, it follows that $A$ and 
$B$ are closed in $\mathcal{X}$ $\times $ $\mathcal{X}$ \cite{A, SAN}, and
so $A$ and $B$ are range space relations. Also, since $A$ and $B$ are
self-adjoint linear relations, it follows that $\ker (A)$, mul$(A)$, $\ker
(B)$ and mul$(B)$ are closed subspace in $\mathcal{X}$ \cite[p.29]{SAN}.
Then the statement $(1)$ follows from Theorem \ref{T1} and the statement $%
(2) $ follows from Theorem \ref{T2}.
\end{proof}

\begin{proposition}
\label{P21}Let $\mathcal{X}$ be a separable Hilbert space and let $%
A\subseteq $ $\mathcal{X}$ $\times $ $\mathcal{X}$ and $B\subseteq $ $%
\mathcal{X}$ $\times $ $\mathcal{X}$ be two range space relations such that
dom$\left( A\right) \ $and ran$\left( B\right) $ are closed. Then:

\begin{enumerate}
\item There is a bounded linear operator $T\subseteq \mathcal{X}$ $\times $ $%
\mathcal{X}$ such that $A^{\ast }=B^{\ast }T$ if and only if $\ker \left(
B\right) \subseteq \ker \left( A\right) $ and dom$\left( A\right) =\overline{%
\text{dom}\left( B\right) }.$

\item There is a bounded linear operator $T\subseteq \mathcal{X}$ $\times $ $%
\mathcal{X}$ such that $A^{\ast }=TB^{\ast }$ if and only if mul$\left(
B\right) \subseteq \overline{\text{mul}\left( A\right) },$ $\overline{\text{%
ran}\left( A\right) }\subseteq $ran$\left( B\right) $ and there is a range
space operator $T_{0}\subseteq $dom$\left( B\right) ^{\bot }\mathcal{\times }
$dom$\left( A\right) ^{\bot }$ such that dom$\left( T_{0}\right) $ is a
closed subspace of dom$\left( B\right) ^{\bot }$ and ran$\left( T_{0}\right)
=$dom$\left( A\right) ^{\bot }$.
\end{enumerate}
\end{proposition}

\begin{proof}
Since dom$\left( A\right) $ is closed it follows that dom$\left( A^{\ast
}\right) $ is closed \cite[Lemma 6.4.12 (i)]{Sa} and since ran$\left(
B\right) $ is closed it follows that ran$\left( B^{\ast }\right) $ is closed 
\cite[Lemma 6.4.12 (iii)]{Sa}. Also, since ran$\left( B\right) $ is closed
it follows that mul$\left( B\right) $ is closed.

$\left( 1\right) $ Since $A^{\ast }$ is a range space relation such that dom$%
\left( A^{\ast }\right) $, ran$\left( B^{\ast }\right) $ and $\ker \left(
B^{\ast }\right) $ are closed, by Theorem \ref{T1}, there is a bounded
linear operator $T\subseteq \mathcal{X}$ $\times $ $\mathcal{X}$ such that $%
A^{\ast }=B^{\ast }T$ if and only ran$\left( A^{\ast }\right) \subseteq $ran$%
\left( B^{\ast }\right) $ and mul$\left( B^{\ast }\right) =$mul$\left(
A^{\ast }\right) $. But mul$\left( A^{\ast }\right) =$dom$\left( A\right)
^{\bot }$and mul$\left( B^{\ast }\right) =$dom$\left( B\right) ^{\bot }$,
and then dom$\left( A\right) =\overline{\text{dom}\left( B\right) }$.

On the other hand, since ran$\left( A^{\ast }\right) =\ker \left( A\right)
^{\bot }$ and ran$\left( B^{\ast }\right) =\ker \left( B\right) ^{\bot }$%
\cite[Lemma 6.4.12 (iii)]{Sa},$\ $it follows that ran$\left( A^{\ast
}\right) \subseteq $ran$\left( B^{\ast }\right) $ if and only if $\ker
\left( B\right) \subseteq \ker \left( A\right) .$

$\left( 2\right) $ Since dom$\left( A^{\ast }\right) ,$ mul$\left( A^{\ast
}\right) $ and ran$\left( B^{\ast }\right) $ are closed, by Theorem \ref{T2}%
, there is a bounded linear operator $T\subseteq \mathcal{X}$ $\times $ $%
\mathcal{X}$ such that $A^{\ast }=TB^{\ast }$ if and only if dom$\left(
A^{\ast }\right) \subseteq $dom$\left( B^{\ast }\right) ,$ $\ker \left(
B^{\ast }\right) \subseteq \ker \left( A^{\ast }\right) $ and there is a
range space operator $T_{\text{mul}}\subseteq $mul$\left( B^{\ast }\right) 
\mathcal{\times }$mul$\left( A^{\ast }\right) $ such that dom$\left( T_{%
\text{mul}}\right) $ is a closed subspace of mul$\left( B^{\ast }\right) $
and ran$\left( T_{\text{mul}}\right) =$mul$\left( A^{\ast }\right) $. Since
dom$\left( A^{\ast }\right) =$mul$\left( A\right) ^{\bot }$ and dom$\left(
B^{\ast }\right) =$mul$\left( B\right) ^{\bot }$\cite[Lemma 6.4.12 (i)]{Sa}
dom$\left( A^{\ast }\right) \subseteq $dom$\left( B^{\ast }\right) $ if and
only if mul$\left( B\right) \subseteq \overline{\text{mul}\left( A\right) }$
and since ran$\left( A\right) ^{\bot }=\ker \left( A^{\ast }\right) $ and ran%
$\left( B\right) ^{\bot }=\ker \left( B^{\ast }\right)$, $\ker \left(
B^{\ast }\right) \subseteq \ker \left( A^{\ast }\right) $ if and only if $%
\overline{\text{ran}\left( A\right) }\subseteq $ran$\left( B\right) $,

Since mul$\left( A^{\ast }\right) =$dom$\left( A\right) ^{\bot }$ and mul$%
\left( B^{\ast }\right) =$dom$\left( B\right) ^{\bot }$, it folows that $%
T_{0}=T_{\text{mul}}.$
\end{proof}

\begin{corollary}
Let $\mathcal{X}$ be a separable Hilbert space and let $A\subseteq $ $%
\mathcal{X}$ $\times $ $\mathcal{X}$ and $B\subseteq $ $\mathcal{X}$ $\times 
$ $\mathcal{X}$ be two self-adjoint linear relations such that dom$\left(
A\right) \ $and ran$\left( B\right) $ are closed. Then:

\begin{enumerate}
\item There is a bounded linear operator $T\subseteq \mathcal{X}$ $\times $ $%
\mathcal{X}$ such that $A=BT$ if and only if $\ker \left( A\right) \subseteq 
$ker$\left( B\right) $ and dom$\left( B\right) =$dom$\left( A\right) .$

\item There is a bounded linear operator $T\subseteq \mathcal{X}$ $\times $ $%
\mathcal{X}$ such that $A=TB$ if and only if mul$\left( B\right) \subseteq $%
mul$\left( B\right) $, ran$\left( A\right) \subseteq $ran$\left( B\right) $
and there is a range space operator $T_{0}\subseteq $dom$\left( B\right)
^{\bot }\mathcal{\times }$dom$\left( A\right) ^{\bot }$ such that dom$\left(
T_{0}\right) $ is a closed subspace of dom$\left( B\right) ^{\bot }$ and ran$%
\left( T_{0}\right) =$dom$\left( A\right) ^{\bot }$.
\end{enumerate}
\end{corollary}

\begin{proof}
Since $A$ and $B$ are self-adjoint linear relations, it follows that $\ker
(A)$, mul$(A)$, $\ker (B)$ and mul$(B)$ are closed subspace in $\mathcal{X}$ 
\cite{SAN}. The rest of the proof follows from Proposition \ref{P21}.
\end{proof}


\begin{thebibliography}{9}
\bibitem{A} R. Arens, \textit{Operational calculus of linear relations, }%
Pacific J. Math. \textbf{9}(1961), 9-23.

\bibitem{D} R.G. Douglas, \textit{On majorization, factorization, and range
inclusion of operators on Hilbert space}, Proc. Amer. Math. Soc., \textbf{17}%
(1966), 413-415.

\bibitem{FW} P.A. Filmore and J.P. Williams, \textit{On operator range, }%
Adv. Math., \textbf{7}(1971), 254-281.

\bibitem{Sa} A. Sandovici, \textit{Habilitation Thesis: Algebraical,
Geometrical and Spectral Theory of Linear Relations} (2018).

\bibitem{SAN} A. Sandovici, \textit{A range matrix-type criterion for the
self adjointness of symmetriclinear relations, }Acta. Math. Hungar. \textbf{%
158}(2019),1,27-35.

\bibitem{SS} A. Sandovici and\ Z. Sebesty\'{e}n, \textit{On operator
factorization of linear relations, }Positivity, \textbf{17}(2013), 1115-1122.

\bibitem{PS} D. Popovici and Z. Sebesty\'{e}n, \textit{Factorization of
linear relations, }Adv. Math., \textbf{233}(2013), 20-55.
\end{thebibliography}
\end{document}